\documentclass[12pt]{article}
\usepackage[utf8]{inputenc}
\usepackage{url}
\usepackage{hyperref}
\usepackage{amsmath}
\usepackage{graphicx}
\usepackage{float}
\usepackage{longtable}
\usepackage[utf8]{inputenc}
\usepackage{amssymb}
\usepackage{amsfonts}
\usepackage{hyperref}
\usepackage{mathrsfs}
\usepackage{mathpazo}
\usepackage{palatino}
\usepackage[table,xcdraw]{xcolor}
\usepackage{xcolor}
\usepackage{enumerate, textcomp}
\usepackage{cite}
\usepackage{amsthm}
\usepackage{pifont}
\usepackage{hyperref}
\usepackage[square, comma, numbers, sort & compress]{natbib}
\usepackage{interval}
\intervalconfig {soft open fences ,separator symbol =,}

\theoremstyle{plain}
\newtheorem{defn}{Definition}[section]

\newtheorem{lem}{Lemma}[section]
\newtheorem{theorem}{Theorem}[section]
\newtheorem{remark}{Remark}[section]
\newtheorem{eg}{Example}[section]

\usepackage{amsmath}
\usepackage{amsfonts}
\usepackage{amssymb}
\usepackage{enumitem}
\date{} 
\usepackage{graphics}

\begin{document}
\pagestyle{myheadings}
	\markboth{{\small\rm \hfill Some Coupled Fixed Point Theorems for $(\psi, \phi)$- contraction with Applications to Fractals
			\hfill}\hspace{-\textwidth}
		\underline{${{}_{}}_{}$\hspace{\textwidth}}}
	{\underline{${{}_{}}_{}$\hspace{\textwidth}}\hspace{-\textwidth}%
		{\small\rm \hfill Some Coupled Fixed Point Theorems for $(\psi, \phi)$- contraction with Applications to Fractals
			\hfill}}
	\thispagestyle{plain}
	
	\begin{center}
		
		{\huge Some Coupled Fixed Point Theorems for $(\psi, \phi)$- contraction with Applications to Fractals 
			\rule{0mm}{6mm}\renewcommand{\thefootnote}{}
			\footnotetext{\scriptsize
				$^*$corresponding author.  e-mail: rameshkumard14@gmail.com\\}
		}
		
		\vspace{1cc}
		{\large\it Athul  P $^{a}$, D. Ramesh Kumar$^{b,*}$}
		\vspace{1cc}
		\begin{center}
			$^{a,b}${\small \textit{Department of Mathematics, School of Advanced Sciences, Vellore Institute of Technology,\\
					Vellore-632014, TN, India}}\\
		\end{center}
		\vspace{1cc}
		\begin{abstract}
			In this paper we obtain coupled fixed point theorem for $(\psi, \phi)$-contractions under some generalized conditions on the real valued functions $\psi$ and $\phi$ defined on $(0, \infty)$. Also, we present a generalized version of coupled fixed point theorem for the same $(\psi, \phi)$-contractions. A new approach to fractal generation by using the relation between fractals and fixed points is given in the light of these fixed point theorems. We establish a new type of iterated function system consisting of generalized $(\psi, \phi)$-contractions. We also extend those results to coupled fractals.
		\end{abstract}
		
		%

	\end{center}
	\hrule
	
	\vspace{1cc}

\section{Introduction and Preliminaries}
Banach contraction principle, proved in 1922 by Banach, is a famous fixed point theorem. This theorem has got many applications in different branches of mathematics as well as other branches of science and technology. Many mathematicians generalized Banach contraction principle in different ways. They led to exciting results in fixed point theory. Boyd and Wong \cite{5} came up with an extension of Banach contraction principle in 1969.
\begin{theorem}(Boyd and Wong \cite{5})
Let $(X, d)$ be a complete metric space and $T: X\rightarrow X$ be a self mapping satisfies $$d\left(T(x), T(y)\right)\leq \phi\left(d(x, y)\right)\ \text{ for each } x, y\in X,$$
where $\phi:\mathbb{R}^+\rightarrow[0, \infty)$ is upper semi-continuous from the right and satisfies the condition $0\leq \phi(t)< t$ for $t> 0$. Then $T$ has a unique fixed point $x_0$ in $X$ and the iterative sequence $\{T^n(x)\}$ converges to $x_0$ for any $x\in X$.
\end{theorem}
Later in 1975 Matkowski \cite{8} proved another variant of Boyd- Wong fixed point theorem. In this variant, the continuity of the function $\phi$ is replaced with some more general condition.
\begin{theorem}(Matkowski \cite{8})
Let $(X, d)$ be a complete metric space and $T: X\rightarrow X$ satisfies $$d\left(T(x), T(y)\right)\leq \phi\left(d(x, y)\right)\ \text{ for each } x, y\in X,$$
where $\phi:(0, \infty)\rightarrow(0, \infty)$ is nondecrasing and satisfies the condition $\lim\limits_{n\rightarrow\infty}\phi^n(t)= 0$ for $t> 0$. Then $T$ has a unique fixed point $x_0$ in $X$ and the iterative sequence $\{T^n(x)\}$ converges to $x_0$ for any $x\in X$.
\end{theorem}
One of the recent generalizations of Banach contraction principle was given by Petko D. Proinov \cite{15} for generalized $(\psi, \phi)$- contractions. We will discuss some of the ideas introduced by P. D. Proinov.
\begin{defn}\cite{15}
Let $T$ be a self mapping on a metric space $(X,d)$. Then it is said to be $(\psi, \phi)$-contraction if it satisfies the contractive-type condition \begin{equation}\label{eq:1.1}
    \psi\left(d\left(Tx, Ty\right)\right)\leq \phi\left(d\left(x, y\right)\right),\ \ \ for\ all\ x,y\in X \ with\ d\left(Tx, Ty\right)> 0,
\end{equation}  where $\psi, \phi :(0,\infty)\rightarrow \mathbb{R}$ are two functions such that $\phi(t)<\psi(t)$ for $t>0$.       
\end{defn}
The main result given by P. D. Proinov is,
\begin{theorem}\label{thm:1.1} \cite{15}
Let $(X,d)$ be a complete metric space and $T:X\rightarrow X$ be a mapping satisfying condition (\ref{eq:1.1}), where the functions $\psi, \phi :(0, \infty)\rightarrow\mathbb{R}$ satisfying the following conditions:
\begin{enumerate}[label=(\roman*)]
\item $\psi$ is nondecreasing;\item $\phi(t)<\psi(t)$ for any $t>0$;\item  $\limsup\limits_{t\rightarrow \epsilon+}\phi(t)< \psi(\epsilon+).$
\end{enumerate} Then $T$ has a unique fixed point $\xi\in X$ and the iterative sequence $\{T^nx\}$ converges to $\xi$ for every $x\in X$.
\end{theorem}
In 2021, O. Popescu \cite{14} proved another generalization by modifying and improving some results proved by P. D. Proinov.
\begin{theorem}\label{thm:1.2}\cite{14}
Let $(X,d)$ be a complete metric space and $T:X\rightarrow X$ be a mapping satisfying condition (\ref{eq:1.1}), where the functions $\psi, \phi :(0, \infty)\rightarrow\mathbb{R}$ satisfying the following conditions:
\begin{enumerate}[label=(\roman*)]
\item $\psi$ is nondecreasing;
\item $\inf\limits_{t>\epsilon}\psi(t)>-\infty$ for any $\epsilon>0$;
\item if $\{\psi(t_n)\}$ and $\{\phi(t_n)\}$ are convergent sequences with the same limit  and $\{\psi(t_n)\}$ is strictly decreasing, then $t_n\rightarrow 0$ as $n\rightarrow \infty$;
\item $\limsup\limits_{t\rightarrow \epsilon+}\phi(t)< \liminf\limits_{t\rightarrow \epsilon+}\psi(t)$ for any $\epsilon>0$;
\item$T$ has a closed graph or $\limsup\limits_{t\rightarrow 0+}\phi(t)<\min\{\liminf\limits_{t\rightarrow\epsilon}\psi(t), \ \psi(\epsilon)\}$ for any $\epsilon>0$.
\end{enumerate}
Then $T$ has a unique fixed point $\xi\in X$ and the iterative sequence $\{T^nx\}$ converges to $\xi$ for every $x\in X$.
\end{theorem}

In this paper, we extend the above fixed point theorems by P. D. Proinov and O. Popescu to coupled fixed point problems which usually discuss about the fixed points of maps of the kind $T:X\times X\rightarrow X$ where $X$ is a complete metric space. Also, we prove a generalized version of the coupled fixed point theorem by replacing the product of a complete metric space by a product of two different complete metric spaces, say $X\times Y$. As a major application of fixed point theory, we extend our work to the theory of fractals too. Here we present a new method to construct fractals using generalized $(\psi, \phi)$-contractions, for which we use the idea of new iterated function system consisting of generalized $(\psi, \phi)$-contractions. This is different from the classical way of generating fractals given by Hutchinson and Barnsley \cite{4}, which uses Banach contraction principle. Towards the end of the paper, inspired from the coupled fixed point theorems, we establish the existence and uniqueness of coupled self-similar sets for generalized $(\psi, \phi)$-contraction mappings. 

\section{Coupled fixed point theorem for $(\psi, \phi)$-contraction}
In this section, we give our main results which extend fixed point theorems by P. D. Proinov and O. Popescu to coupled fixed point problems.
\begin{defn}
Let $(X, d)$ be a metric space and $\psi, \phi :(0, \infty)\rightarrow\mathbb{R}$ be two maps. A map $T:X\times X\rightarrow X$ is said to be jointly $(\psi, \phi)$-contraction if for $z=(x,y), w=(u,v) \in X$, with $\max\{d\left(T\left(x,y\right), T\left(u,v\right)\right), d\left(T\left(y,x\right), T\left(v,u\right)\right)\}>0$, then
\begin{equation}\label{eq:2.1}
    \psi\left(d\left(Tz, Tw\right)\right)\leq\phi\left(\max\{d(x,u), d(y,v)\}\right)
\end{equation}
\end{defn}
For a complete metric space $(X,d)$ we define $X^*=X\times X$. \\ Define $d^*:X^*\times X^*\rightarrow\mathbb{R}$ such that $d^*\left((x,y),(u,v)\right)= \max\{d(x,u), d(y, v)\}.$
It can easily be seen that if $(X,d)$ is a complete metric space then $(X^*,d^*)$ is also a complete metric space.\\ For the map $T:X\times X\rightarrow X$, we can define the iterative sequence as follows:
$$\begin{cases}T^2(x,y) &=T\left(T(x,y), T(y,x)\right)\\ T^3(x,y) &=T\left(T^2(x,y), T^2(y,x)\right)\\  &\vdots\\ T^n(x,y) &=T\left(T^{n-1}(x, y), T^{n-1}(y, x)\right)\end{cases}$$
\begin{theorem}\label{thm:2.1}
Let $(X,d)$ be a complete metric space and $T: X\times X\rightarrow X$ satisfies condition (\ref{eq:2.1}) where $\psi, \phi: (0,\infty)\rightarrow\mathbb{R}$ satisfying the conditions:
\begin{enumerate}[label=(\roman*)]
\item $\psi$ is nondecreasing;
\item $\phi(t)< \psi(t)$ for every $t>0$;
\item$\limsup\limits_{t\rightarrow\epsilon+}\phi(t)< \psi(\epsilon+)$ for every $\epsilon>0$.
\end{enumerate}
Then there exists a unique point $\xi^*=(x^*, y^*)\in X\times X$ such that\[ \begin{cases}x^*= T(x^*, y^*)\\ y^*= T(y^*, x^*)\end{cases} \] and the iterative sequences $x_n= T^n(x, y)$ and $y_n= T^n(y, x)$ converge to $x^*$ and $y^*$ respectively, for any $(x, y)\in X\times X$.
\end{theorem}
\begin{proof}
Since $(X, d)$ is a complete metric space, from the earlier discussion, we have $(X^*, d^*)$ is also a complete metric space. Now we define a map $T^*: X^*\rightarrow X^*$ such that $T^*(x, y)= \left(T(x, y), T(y, x)\right)$ for $(x, y)\in X^*.$
Let $z= (x, y),\ w= (u, v) \in X^*$ with $d^*(T^*z, T^*w)> 0$.
If $d^*(T^*z, T^*w)= \max\{d(T(x,y), T(u,v)), d\left(T(y,x), T(v,u)\right)\}= d\left(T(x,y), T(u, v)\right)$, then from condition (\ref{eq:2.1}) we get,  \begin{align*}
  \psi\left(d^*\left(T^*z, T^*w\right)\right) &= \psi\left(d\left(T(x, y), T(u, v)\right)\right)\\ &\leq\phi\left(\max\{d(x, u), d(y, v)\}\right)\\ &=\phi\left(d^*(z, w)\right) . 
\end{align*}
On the other hand, assume $d^*(T^*z, T^*w)=\max\{d(T(x,y), T(u,v)), d(T(y,x), T(v,u))\}= d(T(y,x), T(v, u))$. Proceeding as above, we get $\psi\left(d^*\left(T^*z, T^*w\right)\right)\leq\phi\left(d^*(z, w)\right).$
Thus it follows that, for any $z,\ w \in X^*$ with $d^*(T^*z, T^*w)>0$, we have $\psi\left(d^*\left(T^*z, T^*w\right)\right)\leq\phi\left(d^*(z, w)\right)$, which means the self mapping $T^*$ satisfies condition (\ref{eq:1.1}) in the complete metric space $(X^*, d^*)$.
Then, from condition \textit{(i)- (iii)} in the hypothesis and Theorem \ref{thm:1.1}, we can conclude that there exists a unique $\xi^*= (x^*, y^*)\in X^*$ such that $T^*\xi^*=\xi^*$.
That is, $\left(T(x^*, y^*), T(y^*, x^*)\right)= (x^*, y^*)$,
which implies that $x^*= T(x^*, y^*), \ y^*= T(y^*, x^*).$
Also for any $z=(x, y)\in X^*$, the iterative sequence $\{{T^*}^n(z)\}$ converges to $\xi^*$. That is, the iterative sequences $x_n= T^n(x,y)$ and $y_n= T^n(y, x)$ converge to $x^*$ and $y^*$ respectively for any $(x, y)\in X\times X$.
This completes the proof.
\end{proof}
In the same way we can have an extension of Theorem \ref{thm:1.2} for coupled fixed points.
\begin{theorem}\label{thm:2.2}
Let $(X,d)$ be a complete metric space and $T: X\times X\rightarrow X$ be a map satisfying the condition (\ref{eq:2.1}) where $\psi, \phi: (0,\infty)\rightarrow\mathbb{R}$ satisfying the conditions:
\begin{enumerate}[label=(\roman*)]
\item $\phi(t)< \psi(t)$ for any $t>0$;
\item $\inf\limits_{t>\epsilon}\psi(t)>-\infty$ for any $\epsilon>0$;
\item if $\{\psi(t_n)\}$ and $\{\phi(t_n)\}$ are convergent sequences with the same limit  and $\{\psi(t_n)\}$ is strictly decreasing, then $t_n\rightarrow 0$ as $n\rightarrow \infty$;
\item $\limsup\limits_{t\rightarrow \epsilon+}\phi(t)< \liminf\limits _{t\rightarrow \epsilon+}\psi(t)$ for any $\epsilon>0$;
\item $T$ has a closed graph or $\limsup\limits_{t\rightarrow 0+}\phi(t)<\min\{\liminf\limits_{t\rightarrow\epsilon}\psi(t), \psi(\epsilon)\}$ for any $\epsilon>0$.\end{enumerate}
Then there exists a unique point $\xi^*=(x^*, y^*)\in X\times X$ such that\[ \begin{cases}x^*= T(x^*, y^*)\\ y^*= T(y^*, x^*)\end{cases} \] and the iterative sequences $x_n= T^n(x, y)$ and $y_n= T^n(y, x)$ converge to $x^*$ and $y^*$ respectively, for any $(x, y)\in X\times X$.
\end{theorem}
\begin{proof}
Define a map $T^*: X^*\rightarrow X^*$ such that $T^*(x, y)= \left(T(x, y), T(y, x)\right) \ \ \text{for } \ (x, y)\in X^*$. 
Then from the proof of Theorem \ref{thm:2.1}, it is clear that $T^*$ satisfies condition (\ref{eq:1.1}).\\ 
Suppose that $T$ has a closed graph. Then for any sequences $\{x_n\}, \{y_n\}$ in $X$ such that $x_n\rightarrow x$, $y_n\rightarrow y$, $T(x_n, y_n)\rightarrow\alpha$ and $T(y_n, x_n)\rightarrow\beta$ as $n\rightarrow\infty$ we have, $T(x, y)= \alpha$ and $T(y, x)= \beta$.
From this, we can conclude that for any sequence $\{(x_n, y_n)\}\subset X\times X$ with $(x_n, y_n)\rightarrow (x,y)$ and $T^*(x_n, y_n)\rightarrow (\alpha, \beta)$ we get, $T^*(x, y)= (\alpha, \beta)$, which means $T^*$ has a closed graph.
Thus if graph of $T$ is closed then graph of $T^*$ is also closed.
Hence the result follows from the conditions \textit{(i)- (v)} in the hypothesis and Theorem \ref{thm:1.2}.
\end{proof}
\begin{eg}
Let $M= \Big\{\frac{1}{2^n}: n\in\mathbb{Z}^+\cup\{0\}\Big\}$ and  $X= M\cup \{0\}$. Let $d: X\times X\rightarrow \mathbb{R}$ be defined as $d(x, y)=|x- y|$. It can be easily verified that $(X, d)$ is a complete metric space. Also, consider the complete metric space $X\times X$ with the maximum metric given by $$d^*\left((x, y), (u, v)\right)=\max\Big\{\left|x-u\right|, \left|y- v\right|\Big\}\ \text{for any }(x, y), (u, v)\in X\times X.$$ Define a map $T: X\times X\rightarrow X$ such that 
\begin{equation*}
    T(x, y)= 
    \begin{cases}
    \frac{1}{2^{\min\{m, n\}+1}} & \text{if } (x, y)\in M\times M \\
    \frac{1}{2^{n+1}} & \text{if } (x, y)\in (M\times\{0\})\cup (\{0\}\times M)\\
    0 & \text{if }(x, y)= (0, 0)
    \end{cases}
\end{equation*}
Also, define $\psi, \phi: (0, \infty)\rightarrow\mathbb{R}$ as follows:
\begin{equation*}
    \psi(t)=
    \begin{cases}
    \frac{t}{2} & \text{if } t\in(0, \frac{1}{2})\\
    \frac{3t}{2} &\text{if }t\in [\frac{1}{2}, 1)\\
    3t & \text{if }t\geq 1
    \end{cases}\ \ \ \text{and}\ \ \
    \phi(t)=
    \begin{cases}
    \frac{t}{4} & \text{if } t\in (0, \frac{1}{2})\\
    t & \text{if } t\in [\frac{1}{2}, 1)\\
    2t &\text{if }t\geq 1
    \end{cases}
\end{equation*}
Here the functions $\psi$ and $\phi$ satisfy the conditions \textit{(i)- (v)} of Theorem \ref{thm:2.2}.\\
\textbf{Case 1}: For $m, n, p,q\geq 0$ with $\min\{m, n\}\neq \min\{p, q\}$, we have
\[\begin{split}\psi\left(d\left(T\left(\frac{1}{2^m}, \frac{1}{2^n}\right), T\left(\frac{1}{2^p}, \frac{1}{2^q}\right)\right)\right)&=\psi\left(\left|\frac{1}{2^{\min\{m, n\}+1}}-\frac{1}{2^{\min\{p, q\}+1}}\right|\right)\\ &= \frac{1}{4}\left|\frac{1}{2^{\min\{m, n\}}}-\frac{1}{2^{\min\{p, q\}}}\right|
\end{split}\]and
\[\begin{split}\phi\left(d^*\left(\left(\frac{1}{2^m}, \frac{1}{2^n}\right), \left(\frac{1}{2^p}, \frac{1}{2^q}\right)\right)\right)&= \phi\left(\max\left\{\left|\frac{1}{2^m}-\frac{1}{2^p}\right|, \left|\frac{1}{2^n}-\frac{1}{2^q}\right|\right\}\right)\\
&=\begin{cases}
\frac{1}{4}\max\left\{\left|\frac{1}{2^m}-\frac{1}{2^p}\right|, \left|\frac{1}{2^n}-\frac{1}{2^q}\right|\right\} &\text{if }m, n, p, q> 0\\
\max\left\{\left|\frac{1}{2^m}-\frac{1}{2^p}\right|, \left|\frac{1}{2^n}-\frac{1}{2^q}\right|\right\} &\text{otherwise}.
\end{cases}
\end{split}\]
\textbf{Case 2}: For $m, n, p\geq 0$ with $\min\{m, n\}\neq p$, we get
\[\psi\left(d\left(T\left(\frac{1}{2^m}, \frac{1}{2^n}\right), T\left(\frac{1}{2^p}, 0\right)\right)\right)= \psi\left(\left|\frac{1}{2^{\min\{m, n\}+1}}-\frac{1}{2^{p+1}}\right|\right)= \frac{1}{4}\left|\frac{1}{2^{\min\{m, n\}}}-\frac{1}{2^{p}}\right|
\]and
\[\begin{split}\phi\left(d^*\left(\left(\frac{1}{2^m}, \frac{1}{2^n}\right), \left(\frac{1}{2^p}, 0\right)\right)\right)&= \phi\left(\max\left\{\left|\frac{1}{2^m}-\frac{1}{2^p}\right|, \frac{1}{2^n}\right\}\right)\\
&=\begin{cases}
\frac{1}{4}\max\left\{\left|\frac{1}{2^m}-\frac{1}{2^p}\right|, \frac{1}{2^n}\right\} &\text{if }\max\left\{\left|\frac{1}{2^m}-\frac{1}{2^p}\right|, \frac{1}{2^n}\right\}<\frac{1}{2}\\
\max\left\{\left|\frac{1}{2^m}-\frac{1}{2^p}\right|, \frac{1}{2^n}\right\} &\text{if }\max\left\{\left|\frac{1}{2^m}-\frac{1}{2^p}\right|, \frac{1}{2^n}\right\}\geq \frac{1}{2}\\
2 &\text{if }n= 0.
\end{cases}
\end{split}\]
\textbf{Case 3}: If $m, p\geq 0$ and $m\neq p$, we get
\[\psi\left(d\left(T\left(\frac{1}{2^m}, 0\right), T\left(\frac{1}{2^p}, 0\right)\right)\right)= \psi\left(\left|\frac{1}{2^{m+1}}-\frac{1}{2^{p+1}}\right|\right)= \frac{1}{4}\left|\frac{1}{2^{m}}-\frac{1}{2^{p}}\right|
\]and
\[\begin{split}
\phi\left(d^*\left(\left(\frac{1}{2^m}, 0\right), \left(\frac{1}{2^p}, 0\right)\right)\right)&=\phi\left(\max\left\{\left|\frac{1}{2^m}-\frac{1}{2^p}\right|, 0\right\}\right)\\
&=\begin{cases}
\frac{1}{4}\left|\frac{1}{2^m}-\frac{1}{2^p}\right| &\text{if }\left|\frac{1}{2^m}-\frac{1}{2^p}\right|<\frac{1}{2}\\
\left|\frac{1}{2^m}-\frac{1}{2^p}\right| &\text{if }\left|\frac{1}{2^m}-\frac{1}{2^p}\right|\geq \frac{1}{2}.
\end{cases}
\end{split}\]On the other hand
\[\psi\left(d\left(T\left(\frac{1}{2^m}, 0\right), T\left(0, \frac{1}{2^p}\right)\right)\right)= \psi\left(\left|\frac{1}{2^{m+1}}-\frac{1}{2^{p+1}}\right|\right)= \frac{1}{4}\left|\frac{1}{2^{m}}-\frac{1}{2^{p}}\right|
\]and
\[\begin{split}
\phi\left(d^*\left(\left(\frac{1}{2^m}, 0\right), \left(0, \frac{1}{2^p}\right)\right)\right)&=\phi\left(\max\left\{\frac{1}{2^m}, \frac{1}{2^p}\right\}\right)\\
&=\begin{cases}
\frac{1}{4}\max\left\{\frac{1}{2^m}, \frac{1}{2^p}\right\} &\text{if }\max\left\{\frac{1}{2^m}, \frac{1}{2^p}\right\}<\frac{1}{2}\\
\frac{1}{2} &\text{if }\max\left\{\frac{1}{2^m}, \frac{1}{2^p}\right\}= \frac{1}{2}\\
2 &\text{if }\max\left\{\frac{1}{2^m}, \frac{1}{2^p}\right\}= 1.
\end{cases}
\end{split}\]
\textbf{Case 4}: For $m, n> 0$, we have
\[\psi\left(d\left(T\left(\frac{1}{2^m}, \frac{1}{2^n}\right), T(0, 0)\right)\right)= \psi\left(\frac{1}{2^{\min\{m, n\}+1}}\right)= \frac{1}{4}\left(\frac{1}{2^{\min\{m, n\}}}\right)
\]and
\[\begin{split}
\phi\left(d^*\left(\left(\frac{1}{2^m}, \frac{1}{2^n}\right), (0, 0)\right)\right)&=\phi\left(\max\left\{\frac{1}{2^m}, \frac{1}{2^n}\right\}\right)\\
&= \begin{cases}
\frac{1}{4}\max\left\{\frac{1}{2^m}, \frac{1}{2^n}\right\} &\text{if }\max\left\{\frac{1}{2^m}, \frac{1}{2^n}\right\}< \frac{1}{2}\\
\frac{1}{2} &\text{if }\max\left\{\frac{1}{2^m}, \frac{1}{2^n}\right\}=\frac{1}{2}.
\end{cases}
\end{split}\]
\textbf{Case 5}: For $m> 0$, we get
\[\psi\left(d\left(T\left(1, \frac{1}{2^m}\right), T(0, 0)\right)\right)= \psi\left(\frac{1}{2^{\min\{0, m\}+1}}\right)= \psi\left(\frac{1}{2}\right)= \frac{3}{4}
\]and
\[\phi\left(d^*\left(\left(1, \frac{1}{2^m}\right), (0, 0)\right)\right)= \phi\left(\max\left\{1, \frac{1}{2^m}\right\}\right)=\phi(1)= 2
\]Also,
\[\psi\left(d\left(T\left(\frac{1}{2^m}, 0\right), T(0, 0)\right)\right)=\psi\left(\frac{1}{2^{m+1}}\right)=\frac{1}{4}\left(\frac{1}{2^m}\right)
\]and
\[\phi\left(d^*\left(\left(\frac{1}{2^m}, 0\right), (0, 0)\right)\right)= \phi\left(\frac{1}{2^m}\right)=\begin{cases}
\frac{1}{4}\left(\frac{1}{2^m}\right) &\text{if }m> 1\\
\frac{1}{2} &\text{if }m= 1.
\end{cases}
\]
\textbf{Case 6}:
\[\psi\left(d\left(T(1, 0), T(0, 0)\right)\right)=\psi\left(d\left(T(\frac{1}{2^0}, 0), T(0, 0)\right)\right)=\psi\left(\frac{1}{2}\right)= \frac{3}{4}
\]and
\[\phi\left(d^*\left((1, 0), (0, 0)\right)\right)=\phi(1)= 2.
\]
From all the above discussed cases, we can observe that, for any $(x, y), (u, v)\in X\times X$, we have
$$\psi\left(d\left(T(x, y), T(u, v)\right)\right)\leq \phi\left(d^*\left((x, y), (u, v)\right)\right)$$
 Hence $T$ satisfies condition (\ref{eq:2.1}) and all the hypotheses of Theorem \ref{thm:2.2}. It is evident that $(x, y)= (0, 0)$ is the only coupled fixed point of $T$.
\end{eg}
\section{Extended coupled fixed points of $(\psi, \phi)$-contractions}
In this section we will consider more general problem in coupled fixed points.
Here, instead of having a product of same complete metric space, we will deal with a product of two different complete metric spaces. \\
Let $(X,d)$ and $(Y, \rho)$ be two complete metric spaces. We define $Z= X\times Y$ and a function $\mu:Z\rightarrow \mathbb{R}$ such that 
\begin{equation}\label{eqn:3.1}
 \mu(z, w)= \mu\left((x,y), (u, v)\right)= \max\{d(x, u), \rho(y, v)\} \text{ for any } z=(x,y), w= (u, v) \text{ in } Z.   
\end{equation}
We can easily observe that, since $(X, d)$ and $(Y, \rho)$ are complete metric spaces, the metric space $(Z, \mu)$ is also complete. 
\begin{defn}
Let $(X, d)\text{ and } (Y, \rho)$ be two complete metric spaces. Let $T:X\times Y\rightarrow X$ and $S:X\times Y\rightarrow Y$ be two maps. $T$ and $S$ are said to be extended jointly $(\psi, \phi)$-contractions if for $z= (x, y),\ w= (u, v)\in X\times Y$ with $\max\{d\left(T(x, y), T(u, v)\right), \rho\left(S(x, y), S(u, v)\right)\}> 0$, we have
\begin{equation}\label{eq:3.1}
\begin{split}
    \psi(d(Tz, Tw)) &\leq \phi(\max\{d(x,u), \rho(y, v)\})\\
    \psi(\rho(Sz, Sw)) &\leq \phi(\max\{d(x, u), \rho(y, v)\})
\end{split}
\end{equation}where $\psi, \phi:(0, \infty)\rightarrow\mathbb{R}$ are two functions such that $\phi(t)< \psi(t)$ for every $t> 0$.
\end{defn}
\begin{lem}\label{lem:3.1}
Let $(X, d)\text{ and } (Y, \rho)$ be two complete metric spaces.
Let $T:X\times Y\rightarrow X\text{ and } S:X\times Y\rightarrow Y$ be extended jointly $(\psi, \phi)$-contractions, where $\psi, \phi:(0, \infty)\rightarrow\mathbb{R}$. Let $F_{TS}:X\times Y\rightarrow X\times Y$ be defined by $F_{TS}(x, y)=\left(T(x, y),S(x, y)\right)$ for $(x, y)\in X\times Y.$ Then $F_{TS}$ is a $(\psi, \phi)$-contraction.  
\end{lem}
\begin{proof}
Let $Z= X\times Y$ and $\mu$ be the metric defined as in equation (\ref{eqn:3.1}). Note that $(Z, \mu)$ is a complete metric space and assume that $z=(x, y),\ w= (u, v) \in Z$. Then
\begin{equation}\label{eq:3.2}
   \begin{split}
    \psi\left(\mu\left(F_{TS}(z), F_{TS}(w)\right)\right) &= \psi\left(\mu\left(\left(T(x, y), S(x, y)\right), \left(T(u, v), S(u, v)\right)\right)\right)\\ &=\psi\left(\max\{d\left(T(x, y), T(u, v)\right), \rho\left(S(x, y), S(u, v)\right)\}\right)
\end{split}
\end{equation}
Suppose that $\mu\left(F_{TS}(z), F_{TS}(w)\right)= d\left(T(x, y), T(u, v)\right)$. Then by condition (\ref{eq:3.1}), equation (\ref{eq:3.2}) becomes:
\begin{align*}
   \psi\left(\mu\left(F_{TS}(z), F_{TS}(w)\right)\right) &= \psi\left(d\left(T(x, y), T(u, v)\right)\right)\\
   &\leq \phi\left(\max\{d(x, u), \rho(y, v)\}\right)\\ &=\phi\left(\mu(z, w)\right).
\end{align*}
On the other hand, if $\mu\left(F_{TS}(z), F_{TS}(w)\right)= \rho\left(S(x, y), S(u, v)\right)$, then by a similar argument as above we get,$$\mu\left(F_{TS}(z), F_{TS}(w)\right)\leq \phi\left(\mu(z, w)\right).$$
Hence the map $F_{TS}$ is a $(\psi, \phi)$-contraction.
\end{proof}
We define iterative sequences for the maps $T$ and $S$ as follows:
$$\begin{cases}T^2(x,y) &= T\left(T(x,y), S(x,y)\right)\\
S^2(x, y) &= S\left(T(x, y), S(x, y)\right)\\
T^3(x,y) &= T\left(T^2(x,y), S^2(x,y)\right)\\  
S^3(x, y) &= S\left(T^2(x, y), S^2(x, y)\right)\\ &\vdots\\ T^n(x,y) &= T\left(T^{n-1}(x, y), S^{n-1}(x, y)\right)\\
S^n(x, y) &= S\left(T^{n-1}(x, y), S^{n-1}(x, y)\right)
\end{cases}$$where $(x, y)\in X\times Y$.\\
Now we can extend Theorem \ref{thm:1.1} to extended jointly $(\psi, \phi)$-contractions to get a more generalized coupled fixed points.
\begin{theorem}
Let $(X, d), (Y, \rho)$ be two complete metric spaces and $T:X\times Y\rightarrow X$ and $S:X\times Y\rightarrow Y$ be extended jointly $(\psi, \phi)$-contractions.
If the functions $\psi, \phi$ satisfy the conditions:
\begin{enumerate}[label=(\roman*)]
\item $\psi$ is nondecreasing;
\item $\phi(t)< \psi(t)$ for every $t>0$;
\item $\limsup\limits_{t\rightarrow\epsilon+}\phi(t)< \psi(\epsilon+)$ for every $\epsilon>0$,\end{enumerate}
then there exists a unique point $\xi^*= (x^*, y^*)\in X\times Y$ such that,
\[\begin{cases}
x^*= T(x^*, y^*)\\
y^*= S(x^*, y^*)
\end{cases}\]
and the sequences $x_n= T^n(x, y)$ and $y_n= S^n(x, y)$ converge to $x^*$ and $y^*$ respectively for any $(x, y)\in X\times Y$.
\end{theorem}
\begin{proof}
Define a map $F_{TS}:X\times Y\rightarrow X\times Y$ by $F_{TS}(x, y)=\left(T(x, y),S(x, y)\right)$ for $(x, y)\in X\times Y.$ Then by Lemma \ref{lem:3.1}, we have $F_{TS}$ is a $(\psi, \phi)$-contraction.
Hence the conditions \textit{(i)-(iii)} in the hypotheses and Theorem \ref{thm:1.1} establish the result.
\end{proof}
\begin{theorem}\label{thm:3.2}
Let $(X, d), (Y, \rho)$ be two complete metric spaces, $T:X\times Y\rightarrow X$ and $S:X\times Y\rightarrow Y$ be extended jointly $(\psi, \phi)$-contractions. 
If the functions $\psi, \phi$ satisfy the conditions:
\begin{enumerate}[label=(\roman*)]
\item $\phi(t)< \psi(t)$ for any $t>0$;
\item $\inf\limits_{t>\epsilon}\psi(t)>-\infty$ for any $\epsilon>0$;
\item if $\{\psi(t_n)\}$ and $\{\phi(t_n)\}$ are convergent sequences with the same limit  and $\{\psi(t_n)\}$ is strictly decreasing, then $t_n\rightarrow 0$ as $n\rightarrow \infty$;
\item $\limsup\limits_{t\rightarrow \epsilon+}\phi(t)< \liminf\limits_{t\rightarrow \epsilon+}\psi(t)$ for any $\epsilon>0$;
\item $T$ and $S$ have closed graph or $\limsup\limits_{t\rightarrow 0+}\phi(t)<\min\{\liminf\limits_{t\rightarrow\epsilon}\psi(t), \psi(\epsilon)\}$ for any $\epsilon>0$.\end{enumerate}
Then there exists a unique point $\xi^*= (x^*, y^*)\in X\times Y$ such that,
\[\begin{cases}
x^*= T(x^*, y^*)\\
y^*= S(x^*, y^*)
\end{cases}\]
and the sequences $x_n= T^n(x, y)$ and $y_n= S^n(x, y)$ converge to $x^*$ and $y^*$ respectively for any $(x, y)\in X\times Y$.
\end{theorem}
\begin{proof}
Define a map $F_{TS}:X\times Y\rightarrow X\times Y$ by $F_{TS}(x, y)=\left(T(x, y),S(x, y)\right)$ for $(x, y)\in X\times Y.$ Then by Lemma \ref{lem:3.1}, we have $F_{TS}$ is a $(\psi, \phi)$-contraction.
Suppose that the graphs of $T$ and $S$ are closed. Then for any sequence $\{(x_n, y_n)\}$ in $X\times Y$ with $(x_n, y_n)\rightarrow (x, y)$, $T(x_n, y_n)\rightarrow \alpha$ and $S(x_n, y_n)\rightarrow \beta$, we have $T(x, y)= \alpha$ and $S(x, y)= \beta$.
Since $T(x_n, y_n)\rightarrow \alpha$ and $S(x_n, y_n)\rightarrow \beta$, we get $F_{TS}(x_n, y_n)= (T(x_n, y_n), S(x_n, y_n))\rightarrow (\alpha, \beta)$.
Also, $T(x, y)= \alpha$ and $S(x, y)= \beta$ implies $F_{TS}(x, y)= (\alpha, \beta)$.
Hence we are able to conclude that if $\{(x_n, y_n)\}$ in $X\times Y$ with $(x_n, y_n)\rightarrow (x, y)$ and $F_{TS}(x_n, y_n)= \left(T(x_n, y_n), S(x_n, y_n)\right)\rightarrow (\alpha, \beta)$ then $F_{TS}(x, y)= (\alpha, \beta)$.
This implies that $F_{TS}$ has a closed graph if $T$ and $S$ have closed graphs.
With this fact along with conditions \textit{(i)- (v)} in the hypotheses and Theorem \ref{thm:1.2} we can complete the proof.
\end{proof}
Even more generally, suppose the maps $T\text{ and } S$ defined above satisfy the following contractive conditions:
\begin{equation}\label{eq:3.3}
\begin{split}
    \psi(d(Tz, Tw)) &\leq \phi_1(\max\{d(x,u), \rho(y, v)\})\\
    \psi(\rho(Sz, Sw)) &\leq \phi_2(\max\{d(x, u), \rho(y, v)\})
\end{split}
\end{equation}
for $z= (x, y),\ w= (u, v)\in X\times Y$ with $\max\{d\left(T(x, y), T(u, v)\right), \rho\left(S(x, y), S(u, v)\right)\}> 0$, where $\psi, \phi_1, \phi_2:(0, \infty)\rightarrow\mathbb{R}$ are functions such that $\phi_i(t)< \psi(t)$ for $i= 1, 2$ and for every $t> 0$.
Let us define a function $\phi:(0, \infty)\rightarrow\mathbb{R}$ as $\phi(t)= \max\{\phi_1(t), \phi_2(t)\}$ for $t> 0$. Since $\phi_i(t)< \psi(t)$ for $i= 1, 2$ and for every $t> 0$, we get $\phi(t)< \psi(t)$ for every $t> 0$. Then we can have the following Lemma.
\begin{lem}\label{lem:3.2}
Let $(X, d)\text{ and } (Y, \rho)$ be two complete metric spaces and $Z= X\times Y$.
Let $T:Z\rightarrow X\text{ and } S:Z\rightarrow Y$ satisfy condition (\ref{eq:3.3}). Define a map $F_{TS}:Z\rightarrow Z$ by $F_{TS}(x, y)=\left(T(x, y),S(x, y)\right)$ for $(x, y)\in Z.$ Then $F_{TS}$ is a $(\psi, \phi)$-contraction.  
\end{lem}
\begin{proof}
Note that $(Z, \mu)$ is a complete metric space. Let $z=(x, y),\ w= (u, v) \in Z$.
\begin{equation}\label{eq:3.4}
   \begin{split}
    \psi\left(\mu\left(F_{TS}(z), F_{TS}(w)\right)\right) &= \psi\left(\mu\left(\left(T(x, y), S(x, y)\right), \left(T(u, v), S(u, v)\right)\right)\right)\\ &=\psi\left(\max\{d\left(T(x, y), T(u, v)\right), \rho\left(S(x, y), S(u, v)\right)\}\right)
\end{split}
\end{equation}
Suppose $\mu\left(F_{TS}(z), F_{TS}(w)\right)= d\left(T(x, y), T(u, v)\right)$, then by condition (\ref{eq:3.3}) equation (\ref{eq:3.4}) becomes:
\begin{align*}
   \psi\left(\mu\left(F_{TS}(z), F_{TS}(w)\right)\right) &= \psi\left(d\left(T(x, y), T(u, v)\right)\right)\\
   &\leq \phi_1\left(\max\{d(x, u), \rho(y, v)\}\right)\\ &=\phi_1\left(\mu(z, w)\right)\\
   &\leq\phi\left(\mu(z, w)\right).
\end{align*}
On the other hand, if $\mu\left(F_{TS}(z), F_{TS}(w)\right)= \rho\left(S(x, y), S(u, v)\right)$, then again by condition (\ref{eq:3.3}) equation (\ref{eq:3.4}) becomes:
\begin{align*}
   \psi\left(\mu\left(F_{TS}(z), F_{TS}(w)\right)\right) &= \psi\left(\rho\left(S(x, y), S(u, v)\right)\right)\\
   &\leq \phi_2\left(\max\{d(x, u), \rho(y, v)\}\right)\\ &=\phi_2\left(\mu(z, w)\right)\\
   &\leq\phi\left(\mu(z, w)\right).
\end{align*} 
Thus in both cases we have $\mu\left(F_{TS}(z), F_{TS}(w)\right)\leq \phi\left(\mu(z, w)\right)$.
Hence the map $F_{TS}$ is a $(\psi, \phi)$-contraction.
\end{proof}
Our aim is to produce a fixed point theorem for a map $T$ satisfying the contractive condition (\ref{eq:3.3}). Before that  we need to prove the following Lemma.
\begin{lem}\label{lem:3.3}
Let $\phi_1, \phi_2: (0, \infty)\rightarrow \mathbb{R}$ be two functions such that for $t_0> 0$, $\limsup\limits_{t\rightarrow t_0}\phi_i(x)$ exists for $i= 1, 2$. If we define $\phi(t)= \max\{\phi_1(t), \phi_2(t)\}$, then \begin{center}$\limsup\limits_{t\rightarrow t_0}\phi(t)\leq \max\{\limsup\limits_{t\rightarrow t_0}\phi_1(t), \limsup\limits_{t\rightarrow t_0}\phi_2(t)\}$.\end{center}
\end{lem}
\begin{proof}
By the definition and properties of limit supremum of functions, since $t_0> 0$ and $\limsup\limits_{t\rightarrow t_0}\phi_i(t)$ exists, the set$$A_i= \left\{l\in \mathbb{R}: \exists \{t_n\}, t_n> 0, t_n\rightarrow t_0, t_n\neq t_0\text{ for all }n> 0, \text{ such that }\phi_i(t_n)\rightarrow l \right\}$$ is non empty, and also $\limsup\limits_{t\rightarrow t_0}\phi_i(t)= \max A_i$ for $i= 1, 2$.\\
Since $\limsup\limits_{t\rightarrow t_0}\phi_i(t)$ exists, we have $\limsup\limits_{t\rightarrow t_0}\phi(t)$ also exists. Now we define
$$A= \left\{l\in \mathbb{R}: \exists \{t_n\}, t_n> 0, t_n\rightarrow t_0, t_n\neq t_0\text{ for all }n> 0, \text{ such that }\phi_i(t_n)\rightarrow l \right\}.$$ Then we have $A$ is nonempty and $\limsup\limits_{t\rightarrow t_0}\phi_i(t)= \max A$.\\
Let $l\in A$. Then by definition, there exists a sequence $\{t_n\}$ in $(0, \infty)$ with $t_n\rightarrow t_0, t_n\neq t_0 \text{ for all }n$ such that $\phi(t_n)\rightarrow l$. Thus for $\epsilon> 0$ there exists $N\in\mathbb{N}$ such that $$\left|\phi(t_n)- l\right|< \epsilon\ \ \text{whenever }n\geq N.$$
That is, $\left|\max\left\{\phi_1(t_n), \phi_2(t_n)\right\}- l\right|< \epsilon\ \ \text{ whenever }n\geq N.$ 
For $k\in \mathbb{N}$, define
\[N_k(\phi_i)= \left\{n\in\mathbb{N}: \left|\phi_i(t_n)- l\right|<\frac{1}{k}\right\} \ \ \ \text{ for } i= 1, 2.
\]
For each $k$, either $N_k(\phi_1)\text{ or }N_k(\phi_2)$ is nonempty. Also, $N_{k+1}(\phi_i)\subseteq N_k(\phi_i)$ for each $k$ and $i= 1, 2$. Thus for at least one $i= 1, 2$, we get $N_k(\phi_i)\neq\emptyset$ for all $k$ . Without loss of generality, assume that $N_k(\phi_1)\neq\emptyset$ for all $k$. Then by picking $n_k\in N_k(\phi_1)$, we get a subsequence $\{t_{n_k}\}$ of $\{t_n\}$ such that $\left|\phi_1(t_{n_k})- l\right|<\frac{1}{m}$ for all $k\geq m$. This implies that $\phi_1(t_{n_k})\rightarrow l$. Hence $l\in A_1$.
By this, we can conclude that, if $l\in A$ then $l\in A_1\cup A_2$, which yields that $\limsup\limits_{t\rightarrow t_0}\phi(t)= \max A\leq \max\{\max A_1, \max A_2\}= \max\{\limsup\limits_{t\rightarrow t_0}\phi_1(t), \limsup\limits_{t\rightarrow t_0}\phi_2(t)\}$
\end{proof}
In the light of these Lemmas, we can have the following result.
\begin{theorem}
Let $(X, d), (Y, \rho)$ be two complete metric spaces and $T:X\times Y\rightarrow X$ and $S:X\times Y\rightarrow Y$ be two maps satisfying condition (\ref{eq:3.3}). 
If the functions $\psi, \phi_1\ \text{and }\phi_2$ satisfy the conditions:
\begin{enumerate}[label=(\roman*)]
\item $\psi$ is nondecreasing;
\item $\phi_i(t)< \psi(t)$ for every $t>0$ and $i= 1, 2$;
\item $\limsup\limits_{t\rightarrow\epsilon+}\phi_i(t)< \psi(\epsilon+)$ for every $\epsilon> 0$.\end{enumerate}
Then there exists a unique point $\xi^*= (x^*, y^*)\in X\times Y$ such that,
\[\begin{cases}
x^*= T(x^*, y^*)\\
y^*= S(x^*, y^*)
\end{cases}\]
and the sequences $x_n= T^n(x, y)$ and $y_n= S^n(x, y)$ converge to $x^*$ and $y^*$ respectively for any $(x, y)\in X\times Y$.
\end{theorem}
\begin{proof}
We define a map $F_{TS}:X\times Y\rightarrow X\times Y$ by $F_{TS}(x, y)=\left(T(x, y),S(x, y)\right)$ for $(x, y)\in X\times Y.$ Then by Lemma \ref{lem:3.1}, we have $F_{TS}$ is a $(\psi, \phi)$-contraction where $\phi(t)= \max\{\phi_1(t), \phi_2(t)\}\text{ for }t> 0$.
Since $\phi_i(t)< \psi(t)$ for every $t> 0$ and $i= 1, 2$, we have $\phi(t)< \psi(t)$ for every $t> 0$.
Now, from condition \textit{(ii)} in the hypothesis and Lemma  \ref{lem:3.3}, we get $\limsup\limits_{t\rightarrow\epsilon+}\phi(t)< \psi(\epsilon+)$ for every $\epsilon> 0$.
Hence the functions $\psi$ and $\phi$ satisfy conditions \textit{(i)-(iii)} in the hypothesis of Theorem \ref{thm:1.1}. Then the result follows from the Theorem \ref{thm:1.1} and Lemma \ref{lem:3.2}.
\end{proof}

\section{Iterated Function Systems with $(\psi, \phi)-$contractions}
According to M. F. Barnsley \cite{4}, fractals can be mathematically identified as the fixed points of some set maps.\\
Let $(X, d)$ be a complete metric space and $\mathcal{H}(X)$ denote the set of all nonempty compact subsets of $X$.\\ 
For any $A, B\in \mathcal{H}(X)$, we define distance from the set $A$ to the set $B$ as
\begin{align*}
    d(A, B) &= \max\{d(x, B): x\in A\}\\
            &= \max_{x\in A}\min_{y\in B} d(x, y).
\end{align*}
Now we define the Hausdorff distance in $\mathcal{H}(X)$ as follows:
$$h_d(A, B)= \max\{d(A, B), d(B, A)\}.$$
It can be easily verified that $h_d$ defines a metric on $\mathcal{H}(X)$. Moreover we have the following result:
\begin{lem}\label{lem:4.1} \cite{4}
Let $(X, d)$ be a complete metric space. Then $(\mathcal{H}(X), h_d)$ is a complete metric space. Moreover, if $\{A_n\}$ is a Cauchy sequence in $\mathcal{H}(X)$ then,
$$A= \lim_{n\rightarrow\infty}A_n = \{x\in X: \exists\text{ a Cauchy sequence} \ \{x_n\}\ \text{in}\ X \ \text{such that} \ x_n\in A_n\ and\ \lim\limits_{n\rightarrow\infty}x_n=x\}.$$
\end{lem}
The space $(\mathcal{H}(X), h_d)$ is usually called as the space of fractals.
\begin{lem}\label{lem:4.2} \cite{17,18}
If $\{A_i:i=1,2,\dots n\}, \{B_i: i=1,2,\dots n\}$ be two finite collections in $\mathcal{H}(X)$, then $h_d\left(\bigcup\limits_{i=1}^nA_i, \bigcup\limits_{i=1}^nB_i\right)\leq \max\{h_d(A_i, B_i): i=1,2,\dots n\}$.
\end{lem}
We are now ready to prove some results in fractals.
\begin{lem}\label{lem:4.3}
Let $(X, d)$ be a complete metric space and $w: X\rightarrow X$ be a continuous map satisfying condition (\ref{eq:1.1}) with the functions $\psi, \phi$ which are nondecreasing. Then $\hat{w}:\mathcal{H}(X)\rightarrow\mathcal{H}(X)$ defined by $\hat{w}(A)= w(A)= \bigcup\limits_{a\in A}\{w(a)\}$ for $A\in \mathcal{H}(X)$ also satisfies condition (\ref{eq:1.1}) in $(\mathcal{H}(X), h_d)$.
\end{lem}
\begin{proof}
Let $A, B\in \mathcal{H}(X)$ such that $h_d(\hat{w}(A), \hat{w}(B))>0$.
Assume that $h_d\left(\hat{w}(A), \hat{w}(B)\right)= d\left(\hat{w}(A), \hat{w}(B)\right)= \sup\limits_{x\in A}\inf\limits_{y\in B}d\left(w(x), w(y)\right)>0$.
Since $w$ and $d$ are continuous and $A$ is compact, there exist $a\in A$ such that $d\left(\hat{w}(A), \hat{w}(B)\right)= \inf\limits_{y\in B}d\left(w(a), w(y)\right)>0$, which implies that $d\left(w(a), w(y)\right)>0$ for every $y\in B$. Thus for any $y\in B$,
\begin{align*}
    \psi\left(d\left(\hat{w}(A), \hat{w}(B)\right)\right) &= \psi\left(\inf_{y\in B}d\left(w(a), w(y)\right)\right)\\ &\leq \psi\left(d\left(w(a),w(y)\right)\right)\\ &\leq\phi\left(d(a, y)\right).
\end{align*}
Let $b\in B$ be such that $d(a, b)= \inf\limits_{y\in B}d(a, y)$. Since $\phi$ is nondecreasing,
\begin{align*}
    \psi\left(h_d\left(\hat{w}(A), \hat{w}(B)\right)\right) &= \psi\left(d\left(\hat{w}(A), \hat{w}(B)\right)\right)\\
                        &\leq \phi\left(d(a, b)\right)\\
                        &= \phi\left(\inf_{y\in B}d(a, y)\right)\\
                        &\leq \phi\left(\sup_{x\in A}\inf_{y\in B}d(x, y)\right)\\
                        &\leq \phi\left(h_d(A, B)\right).
\end{align*}
On the other hand, if we assume\\ $h_d(\hat{w}(A), \hat{w}(B))= d(\hat{w}(B), \hat{w}(A))= \sup\limits_{x\in B}\inf\limits_{y\in A}d\left(w(x), w(y)\right)>0$, we can proceed as above and get $\psi\left(h_d(\hat{w}(A), \hat{w}(B))\right)\leq \phi\left(h_d(A, B)\right)$. This completes the proof.
\end{proof}
\begin{remark}
The map $\hat{w}$ in Lemma \ref{lem:4.3} is called the fractal operator generated by $w$. Continuity of the map $w$ is required to make sure that $\hat{w}$ maps $\mathcal{H}(X)$ to itself. A fixed point $A^*\in\mathcal{H}(X)$ of the map $\hat{w}$, if it exists, is called an attractor or a self-similar set of $w$.
\end{remark}
\begin{theorem}\label{thm:4.1}
Let $(X, d)$ be a complete metric space and $w:X\rightarrow X$ be a continuous map. If the map $w$ satisfies condition (\ref{eq:1.1}), where the functions $\psi, \phi:(0, \infty)\rightarrow\mathbb{R}$ satisfy the following conditions:
\begin{enumerate}[label=(\roman*)]
\item $\psi, \phi$ are nondecreasing;
\item $\phi(t)<\psi(t)$ for any $t>0$;
\item $\limsup\limits_{t\rightarrow \epsilon+}\phi(t)< \psi(\epsilon+),$\end{enumerate}
then there exists a unique attractor, say $A^*\in \mathcal{H}(X)$, for $w$. Moreover, for any $A\in \mathcal{H}(X)$ the sequence $\{A_n\}$ in $\mathcal{H}(X)$, given by $A_n= w^{n}(A)$, converges to $A^*$.  
\end{theorem}
\begin{proof}
Let $\hat{w}$ be the fractal operator generated by $w$. By Lemma \ref{lem:4.3} it is clear that $\hat{w}$ satisfies condition (\ref{eq:1.1}). Then, conditions \textit{(i)- (iii)} in the hypothesis and Theorem \ref{thm:1.1} guarantee a unique fixed point for $\hat{w}$, say $A^*\in \mathcal{H}(X)$ and moreover, for any $A\in\mathcal{H}(X)$, the sequence $\{A_n\}$, where $A_n= \hat{w}^n(A)$, converges to $A^*$.\\
This proves the theorem.
\end{proof}
Next lemma will prove a result on graph of the function $\hat{w}$.
\begin{lem}\label{lem:4.4}
Let $(X, d)$ be a complete metric space and $w:X\rightarrow X$ be a continuous map. If $\hat{w}$ is the fractal operator generated by $w$, then graph of $\hat{w}$ is closed.
\end{lem}
\begin{proof}
Let $(A_n)$ be a sequence in $\mathcal{H}(X)$ such that $A_n\rightarrow A$ and $\hat{w}(A_n)\rightarrow B$. The proof is complete if we prove $B= \hat{w}(A)$. By Lemma \ref{lem:4.1}, we have
$$A= \lim_{n\rightarrow\infty}A_n = \{x\in X: \exists\text{ a Cauchy sequence} \ \{x_n\}\ \text{in}\ X \ \text{such that} \ x_n\in A_n,\ \lim\limits_{n\rightarrow\infty}x_n=x\}$$and
\[B=\lim_{n\rightarrow\infty}\hat{w}(A_n)=\{x\in X: \exists\text{ a Cauchy sequence} \ \{x_n\}\ \text{such that} \ x_n\in \hat{w}(A_n),\  \lim_{n\rightarrow\infty} x_n= x\}.\]
Let $x\in A$, then there exists a Cauchy sequence $\{x_n\}$ in $X$, with $x_n\in A_n$, such that $x_n\rightarrow x$. Since $w$ is continuous, we get $w(x_n)\rightarrow w(x)$. In other words there is a sequence $\left\{w(x_n)\right\}$ where $w(x_n)\in \hat{w}(A_n)$ such that $w(x_n)\rightarrow w(x)$.
Then by definition of $B,\ \ w(x)\in B$, which implies that $w(A)\subseteq B$.
On the other hand, let $x\in B$, then there exists a Cauchy sequence $\{x_n\}$ in $X$ where $x_n\in \hat{w}(A_n)$ such that $x_n\rightarrow x$. Then for each $n$, there exists $y_n\in A_n$ such that $x_n= w(y_n)$ and hence we can write $w(y_n)\rightarrow x$.
Since $(A_n)$ is a Cauchy sequence in $\mathcal{H}(X)$, for $\epsilon> 0$ there exists an $N\in \mathbb{N}$ such that $$h_d(A_n, A_m)< \frac{\epsilon}{3}\ \ \ \text{ for every}\ \  n, m\geq N.$$
Then we have both $d(A_n, A_m)< \frac{\epsilon}{3}$ and $d(A_m, A_n)< \frac{\epsilon}{3}$ for all $n, m\geq N$.
By definition, $d(A_n, A_m)< \frac{\epsilon}{3}$ implies $d(y_n, A_m)< \frac{\epsilon}{3}$ for every $n, m\geq N$.
Since $A_m$ is compact, there exist $a\in A_m$ such that, $d(y_n, a)= d(y_n, A_m)< \frac{\epsilon}{3}$ for every $n\geq N$.
By a similar argument, for $d(A_m, A_n)<\frac{\epsilon}{3}$ we get a $b\in A_n$ such that $d(y_m, b)< \frac{\epsilon}{3}$ for every $m\geq N$.
From the condition $d(A_n, A_m)< \frac{\epsilon}{3}$, we get $d(a, b)< \frac{\epsilon}{3}$. Therefore,
$$d(y_n, y_m)\leq d(y_n,a)+ d(a,b)+ d(b, y_m)< \epsilon\ \ \ \text{for every}\ \ n, m\geq N,$$ which implies that $\{y_n\}$ is a Cauchy sequence in $X$. Then there exists a $y\in X$ such that $y_n\rightarrow y$ as $n\rightarrow \infty$. Thus we obtain $y\in A$. Also, by the continuity of $w$ we get, $w(y_n)\rightarrow w(y)$. This implies that $x= w(y)$. Thus $x\in \hat{w}(A)$, which proves that $B\subseteq \hat{w}(A)$. This completes the proof.
\end{proof}
As an application of Theorem \ref{thm:1.2}, we can have the following result:
\begin{theorem}\label{thm:4.2}
Let $(X, d)$ be a complete metric space and $w:X\rightarrow X$ be a continuous map satisfying the condition (\ref{eq:1.1}), where the nondecreasing functions $\psi, \phi:(0, \infty)\rightarrow\mathbb{R}$ satisfy the following conditions:
\begin{enumerate}[label=(\roman*)]
\item $\phi(t)< \psi(t)$ for any $t>0$;
\item $\inf\limits_{t>\epsilon}\psi(t)>-\infty$ for any $\epsilon>0$;
\item if $\{\psi(t_n)\}$ and $\{\phi(t_n)\}$ are convergent sequences with the same limit  and $\{\psi(t_n)\}$ is strictly decreasing, then $t_n\rightarrow 0$ as $n\rightarrow \infty$;
\item $\limsup\limits_{t\rightarrow \epsilon+}\phi(t)< \liminf\limits_{t\rightarrow \epsilon+}\psi(t)$ for any $\epsilon>0$.\end{enumerate}
Then there exists a unique attractor, say $A^*\in \mathcal{H}(X)$, for $w$. Moreover, for any $A\in \mathcal{H}(X)$ the sequence $\{A_n\}$ in $\mathcal{H}(X)$, given by $A_n= w^{n}(A)$, converges to $A^*$.
\end{theorem}
\begin{proof}
Let $\hat{w}$ be the fractal operator generated by $w$. By Lemma \ref{lem:4.3}, it is clear that $\hat{w}$ satisfies condition (\ref{eq:1.1}). Then, conditions \textit{(i)- (iv)} in the hypothesis along with Lemma \ref{lem:4.4} and Theorem \ref{thm:1.2} guarantee a unique fixed point for $\hat{w}$, say $A^*\in \mathcal{H}(X)$ and moreover for any $A\in\mathcal{H}(X)$, the sequence $\{A_n\}$, where $A_n= \hat{w}^n(A)$, converges to $A^*$.\\
This proves the theorem.
\end{proof}
Now, instead of a single function $w$ on $X$, we consider an iterated function system (IFS) $\{X; w_1, w_2, \cdots, w_N\}$ where $w_i: X\rightarrow X$ for $i=1,2,\cdots, N$ are continuous and satisfies condition (\ref{eq:1.1}). 
The function $W: \mathcal{H}(X)\rightarrow \mathcal{H}(X)$, defined by $W(A)= \bigcup\limits_{i= 1}^N\hat{w_i}(A)$, is called the fractal operator generated by the IFS $\{X; w_1, w_2, \cdots, w_N\}$. A point $A\in\mathcal{H}(X)$ such that $W(A)=\bigcup\limits_{i= 1}^N\hat{w_i}(A)= A$, a fixed point of $W$, is called an attractor of the IFS.\\
Let us consider a more general situation.
\begin{lem}\label{lem:4.5}
Let $(X, d)$ be a complete metric space and  $\{X; w_1, w_2, \cdots, w_N\}$ be an IFS where $w_i: X\rightarrow X$ are continuous maps satisfy the condition: 
\begin{equation}\label{eq:4.1}
\psi(d(w_i(x), w_i(y)))\leq \phi_i(d(x, y)),\end{equation} where $\psi, \phi_i:(0, \infty)\rightarrow \mathbb{R}$ are nondecreasing for  $i=1,2,\cdots, N$.
If $W$ is the fractal operator generated by the IFS, then it satisfies the condition (\ref{eq:1.1}), where $\phi(t)=\max_{1\leq i\leq N}\phi_i(t)$ for $t\in (0, \infty).$ 
\end{lem}
\begin{proof}
Let $A,B \in\mathcal{H}(X)$. Then, we have $\psi\left(h_d\left(\hat{w_i}(A), \hat{w_i}(B)\right)\right)\leq \phi_i\left(h_d(A, B)\right)$ for $i= 1, 2, \cdots, N$.
By Lemma \ref{lem:4.2}, we have 
\begin{align*}
    0\leq h_d\left(W(A), W(B)\right) &= h_d\left(\bigcup\limits_{i=1}^N\hat{w_i}(A), \bigcup\limits_{i=1}^N\hat{w_i}(B)\right) \\
    &\leq \max_{1\leq i\leq N}h_d\left(\hat{w_i}(A), \hat{w_i}(B)\right)\\
    &= h_d\left(\hat{w_j}(A), \hat{w_j}(B)\right),
\end{align*} for some $j\in \{1, 2,\cdots N\}$. Since both $\psi$ and $\phi_j$ are nondecreasing, we get
\begin{align*}
    \psi\left(h_d\left(W(A), W(B)\right)\right) &\leq \psi\left(h_d\left(\hat{w_j}(A), \hat{w_j}(B)\right)\right)\\
                       &\leq \phi_j\left(h_d(A, B)\right)\\
                       &\leq \phi\left(h_d(A, B)\right).
\end{align*}
Hence the proof.
\end{proof}
As an immediate consequence of this Lemma along with Theorem  \ref{thm:1.1}, we have the following result:
\begin{theorem}\label{thm:4.3}
Let $(X, d)$ be a complete metric space and  $\{X; w_1, w_2, \cdots, w_N\}$ be an IFS where $w_i: X\rightarrow X$ are continuous maps satisfying the condition (\ref{eq:4.1}). If the maps $\psi, \phi_i$, for $i= 1, 2, \cdots, N$, are nondecreasing and satisfy the conditions:
\begin{enumerate}[label=(\roman*)]
\item $\phi_i(t)< \psi(t)$ for every $t>0$;
\item $\limsup\limits_{t\rightarrow\epsilon+}\phi_i(t)< \psi(\epsilon+)$ for every $\epsilon>0$,\end{enumerate}
then there exists a unique attractor, $A^*\in \mathcal{H}(X)$, for the IFS. Moreover, for any $A\in \mathcal{H}(X)$, the iterated sequence $A_n= W^n(A)$ converges to $A^*$. 
\end{theorem}
\begin{proof}
Define a map $\phi:(0, \infty)\rightarrow\mathbb{R}$ such that $\phi(t)= \max_{1\leq i \leq N}\phi_i(t)$ for $t\in (0, \infty)$. Then from condition \textit{(i)}
 in the hypothesis, it is clear that, $$\phi(t)= \max_{1\leq i \leq N}\phi_i(t)\leq \psi(t), \ \ \text{for every}\ t>0.$$
 Now, from condition \textit{(ii)} in the hypothesis and a similar argument in Lemma \ref{lem:3.3} we get, $\limsup\limits_{t\rightarrow\epsilon+}\phi(t)< \psi(\epsilon+)$ for every $\epsilon> 0$.
 Thus the functions $\psi$ and $\phi$ satisfies conditions \textit{(i)- (iii)} of Theorem \ref{thm:1.1}. Then the proof follows immediately from Theorem \ref{thm:1.1} along with Lemma \ref{lem:4.5}.
\end{proof}
Let us now consider an example:
\begin{eg}
Consider the complete metric space $\mathbb{R}$ with Euclidean metric. We define two maps $w_1, w_2:\mathbb{R}\rightarrow\mathbb{R}$ as:
\[w_1(x)=\begin{cases}
\frac{2}{3}x &\text{if } x\geq 0\\
\frac{-2}{3}x &\text{if } x< 0
\end{cases}\ \ \ \text{and}\ \ \ 
w_2(x)=\begin{cases}
\frac{1}{3}x+\frac{2}{3} &\text{if }x\geq 0\\
\frac{-1}{3}x+\frac{2}{3} &\text{if }x< 0
\end{cases}
\]
Define three functions $\psi, \phi_1, \phi_2:(0, \infty)\rightarrow\mathbb{R}$ as:
\[\psi(t)=\begin{cases}
2t &\text{if }0< t\leq 1\\
3t &\text{if } t> 1
\end{cases},\ 
\phi_1(t)=\begin{cases}
\frac{3}{2}t &\text{if }0< t\leq 1\\
2t &\text{if } t> 1
\end{cases}\ \ and\ \
\phi_2(t)= \begin{cases}
t &\text{if }0< t\leq 1\\
\frac{5}{2}t &\text{if } t> 1
\end{cases}
\]
It can be easily observed that the maps $w_1$ and $w_2$ are continuous and $\psi, \phi_1$ and $\phi_2$ are nondecreasing and satisfy the conditions \textit{(i)} and \textit{(ii)} of Theorem  \ref{thm:4.3}.\\
Consider the maps $w_1, \psi \ \text{and }\phi_1$. Let $x, y\in \mathbb{R}$\\
\textbf{Case 1}: For $x, y\geq 0$ or $x, y< 0$, we get
\[\psi\left(d\left(w_1(x), w_1(y)\right)\right)=\psi\left(\left|\frac{2}{3}x-\frac{2}{3}y\right|\right)=\begin{cases}
\frac{4}{3}\left|x- y\right| &\text{if } \left|x- y\right|\leq 1\\
2\left|x- y\right| &\text{if } \left|x- y\right|> 1.
\end{cases}
\]
\textbf{Case 2}: If $x\geq 0$ and $y< 0$ we have
\[\psi\left(d\left(w_1(x), w_1(y)\right)\right)=\psi\left(\left|\frac{2}{3}x+\frac{2}{3}y\right|\right)=\begin{cases}
\frac{4}{3}\left|x+ y\right| &\text{if } \left|x+ y\right|\leq 1\\
2\left|x+ y\right| &\text{if } \left|x+ y\right|> 1.
\end{cases}
\]
Now, for any $x, y\in \mathbb{R}$ we have
\[\phi_1\left(d(x, y)\right)=\begin{cases}
\frac{3}{2}\left|x- y\right| &\text{if } \left|x- y\right|\leq 1\\
2\left|x- y\right| &\text{if } \left|x- y\right|> 1.
\end{cases}
\]Thus from the above cases it is clear that, for any $x, y\in \mathbb{R}$, $$\psi\left(d\left(w_1(x), w_2(y)\right)\right)\leq \phi_1\left(d(x, y)\right).$$
Similarly, if we consider the maps $w_2, \psi \ \text{and }\phi_2$ we can have the following cases.\\
\textbf{Case 1}: For $x, y\geq 0$ or $x, y< 0$, we have
\[\psi\left(d\left(w_2(x), w_2(y)\right)\right)=\psi\left(\left|\frac{1}{3}x-\frac{1}{3}y\right|\right)=\begin{cases}
\frac{2}{3}\left|x- y\right| &\text{if } \left|x- y\right|\leq 1\\
\left|x- y\right| &\text{if } \left|x- y\right|> 1.
\end{cases}
\]
\textbf{Case 2}: If $x\geq 0$ and $y< 0$, we get
\[\psi\left(d\left(w_2(x), w_2(y)\right)\right)=\psi\left(\left|\frac{1}{3}x+\frac{1}{3}y\right|\right)=\begin{cases}
\frac{2}{3}\left|x+ y\right| &\text{if } \left|x+ y\right|\leq 1\\
\left|x+ y\right| &\text{if } \left|x+ y\right|> 1.
\end{cases}
\]
Now, for any $x, y\in \mathbb{R}$ we have
\[\phi_2\left(d(x, y)\right)=\begin{cases}
\left|x- y\right| &\text{if }\left|x- y\right|\leq 1\\
\frac{5}{2}\left|x- y\right| &\text{if } \left|x- y\right|>1.
\end{cases}
\]
Here also, we can observe that, for any $x, y\in \mathbb{R}$, $$\psi\left(d\left(w_2(x), w_2(y)\right)\right)\leq \phi_2\left(d\left(x, y\right)\right).$$
Thus the IFS $\{\mathbb{R}; w_1, w_2\}$ along with the maps $\psi, \phi_1\text{ and }\phi_2$ satisfies the condition (\ref{eq:4.1}) and the hypothesis of Theorem \ref{thm:4.3}. Hence by Theorem \ref{thm:4.3}, the map $W:\mathcal{H}(\mathbb{R})\rightarrow\mathcal{H}(\mathbb{R})$ defined by $W(A)=w_1(A)\cup w_2(A)$ satisfies condition (\ref{eq:1.1}) with the functions $\psi$ and $\phi= \max\{\phi_1, \phi_2\}$. Also we can observe that $A=[0, 1]$ is the unique attractor of this IFS. We have $w_1([0, 1])=[0, \frac{2}{3}]$ and $w_2([0, 1])= [\frac{2}{3}, 1]$. Hence, $W([0, 1])= w_1([0, 1])\cup w_2([0, 1])= [0, 1]$. 
\end{eg}
\section{Applications to Coupled Fractals}
In this section we will discuss about the existence and uniqueness of coupled self- similar sets for jointly $(\psi, \phi)$-contractions.
\begin{theorem}\label{thm:5.1}
Let $(X, d)$ be a complete metric space and $w:X\times X\rightarrow X$ be a continuous map satisfying condition (\ref{eq:2.1}) where the functions $\psi, \phi: (0, \infty)\rightarrow \mathbb{R}$ satisfy the following conditions:
\begin{enumerate}[label=(\roman*)]
\item $\psi, \phi$ are nondecreasing;
\item $\phi(t)< \psi(t)$ for every $t>0$;
\item $\limsup\limits_{t\rightarrow\epsilon+}\phi(t)< \psi(\epsilon+)$ for every $\epsilon>0$.\end{enumerate} Then there exists a unique element $(A^*, B^*)\in \mathcal{H}(X)\times\mathcal{H}(X)$ such that 
\[\begin{cases}
  A^*= w(A^*, B^*) \\
  B^*= w(B^*, A^*)
\end{cases}\] 
Moreover, for any element $(A, B)\in \mathcal{H}(X)\times\mathcal{H}(X)$, the sequences
\[\begin{cases}
    A^n= w^n(A, B)\\
    B^n= w^n(B, A)
\end{cases}\]
converge to $A^*$ and $B^*$ respectively.
\end{theorem}
\begin{proof}
We define the operator $w^*: X^*\rightarrow X^*$ such that $w^*(x, y)= \left(w(x, y), w(y, x)\right)$. By Theorem \ref{thm:2.1}, it is clear that $w^*$ satisfies the condition (\ref{eq:1.1}) in the complete metric space $(X^*, d^*)$. Let $W^*:\mathcal{H}(X^*)\rightarrow\mathcal{H}(X^*)$ be the fractal operator generated by $w^*$. Then Theorem \ref{lem:4.3} tells that the operator $W^*$ satisfies condition (\ref{eq:1.1}) in the complete metric space $(\mathcal{H}(X^*), h_{d^*})$. Hence the result follows from Theorem \ref{thm:4.1}.
\end{proof}
Our next result deals with existence of coupled self- similar set for an IFS.
\begin{theorem}
Let $w_i:X\times X\rightarrow X$ be continuous and jointly $(\psi, \phi_i)$-contractions for $i=1, 2, \cdots, N$ and $\psi, \phi_i$ satisfy the following conditions:
\begin{enumerate}[label=(\roman*)]
\item $\psi, \phi_i$ are nondecreasing for $i= 1, 2, \cdots, N$;
\item $\phi_i(t)< \psi(t)$ for every $t>0$;
\item $\limsup\limits_{t\rightarrow\epsilon+}\phi_i(t)< \psi(\epsilon+)$ for every $\epsilon>0$.\end{enumerate}
Then there exists a unique pair $(A^*, B^*)\in\mathcal{H}(X)\times\mathcal{H}(X)$ such that 
\[\begin{cases}
   A^*= \bigcup\limits_{i=1}^N w_i(A^*, B^*)\\
   B^*= \bigcup\limits_{i=1}^N w_i(B^*, A^*).
\end{cases}\]
Moreover, for any $(A, B)\in \mathcal{H}(X)\times\mathcal{H}(X)$, the sequences
\[\begin{cases}
    A_n= \bigcup\limits_{i=1}^N w_i(A, B)\\
    B_n= \bigcup\limits_{i=1}^N w_i(B, A)
\end{cases}\]
converge to $A^*$ and $B^*$ respectively.
\end{theorem}
\begin{proof}
Define $w_i^*: X^*\rightarrow X^*$ by $w_i^*(x,y)= (w_i(x, y), w_i(y, x))$ for $i=1, 2, \cdots, N$. Then from the proof of Theorem \ref{thm:2.1} we can say each map $w_i^*$ satisfies condition (\ref{eq:1.1}) with functions $\psi, \phi_i$ for $i=1, 2, \cdots, N$. Now we define, the fractal operator generated by the IFS $\{X^*; w_1^*, w_2^*, \cdots, w_N^*\}$,  $W^*:\mathcal{H}(X^*)\rightarrow\mathcal{H}(X^*)$ by $W^*(A, B)=\bigcup\limits_{i=1}^N\hat{w_i^*}(A, B)$. From Lemma \ref{lem:4.5}, it is clear that $W^*$ satisfies condition (\ref{eq:1.1}) with $\phi= \max\limits_{1\leq i\leq N}\phi_i$. Then the proof can be completed by using Theorem \ref{thm:4.3}. \end{proof}

\end{document}